\documentclass[12pt]{article}

\usepackage{amsmath, amsthm, color}
\usepackage{amssymb} 
\usepackage{amscd}

\renewcommand{\r}{\mathbb{R}}

\newtheorem{thm}[equation]{Theorem}
\newtheorem{cor}[equation]{Corollary}

\theoremstyle{definition}

\newtheorem{rem}[equation]{Remark}

\usepackage{polynom}

\title{Zero sets and factorization of polynomials of two variables}
\author{ \ Micki Balaich 
and\ Mihail Cocos
\thanks{Mathematics Subject
Classification: Primary 12D05 ; Secondary 12E05}}

\date{\today}

\begin{document}
\maketitle

\begin{abstract}
The relationship between a polynomial's zeros and factors is well known.  If $a\in\r$ is a zero of $f(x)\in \r[x]$ then $x-a$ is a factor of $f(x)$.  In this paper, we generalize this idea to $\r[x,y]$.  We consider the zero sets of two variable polynomials in $\r[x,y]$ and give criterion to when two polynomials with the same zero set in $\r[x,y]$ have a common factor with the same zero set.  When $F$ is not a field, but the division algebra of Quaternions, we provide an example of two polynomials in $F[x,y]$ with the same zero set and no common factor.  
\end{abstract}

\section{Introduction}

Recall that a {zero} of a polynomial $p(x)\in \r[x]$ is a real number $a$ such that $p(a)=0$.  It is a college algebra fact that if a real number $a$ is a zero of a polynomial $p(x)\in \r[x]$ then $x-a$ divides $p(x)$. Many examples demonstrate this idea.  For instance, if $p(x)={x^2}-1$ then $p$ is divisible by $x-1$ and $x+1$ since the zeros of $f$ are $1$ and $ -1$ and we can write the familiar $p(x)=(x-1)(x+1)$.  This fundamental relationship guarantees that any two polynomials of one variable over the real numbers that share a zero set have a common factor with the same zero set. 

\indent Although this relationship is generally attributed only to one variable polynomials, the same can be said for any two polynomials in $n$ variables over the complex numbers, or any algebraically closed field.  As motivation for our work, we will  now establish this result using the cornerstone of algebraic geometry,  the Hilbert Nullstellensatz Theorem.

\begin{thm} \label{thm: algebraically closed}
Let $F$ be an algebraically closed field.  If $p,  g \in F[x_1, x_2, \ldots, x_n]$  have the same zero set, then $p$ and $g$ have a common factor with the same zero set. 
\end{thm}

\begin{proof}

Assume $p$ and $g$  share the same zero set and let $I=<{g}>$.  Since $F$ is algebraically closed, by the Nullstellensatz  there exists a positive integer $r$ such that ${p}^r=hg$ where $h \in  F[x_1, x_2, \ldots, x_n].$  Since $F[x_1, x_2, \ldots, x_n]$is a UFD $g$ is either irreducible or it can be written as a product of irreducible factors.  If $g$ is irreducible then because $F[x_1, x_2, \ldots, x_n]$ is a UFD it is also prime and must therefore divide $p.$  If $g$ is not irreducible then it can be written as the product of irreducible factors $g_1 \cdots g_k$ for some positive integer $k$.  Now since each of these factors is irreducible they must also be prime, so they each divide $p.$  Then the product of the distinct irreducible factors of $g$ divides $p$.  Thus $p$ and $g$ have a common factor.  Now a repeated factor does not add any new zeros to the zero set of $g$ so the common factor will have the same zero set as $p$ and $g.$  It follows that $p$ and $g$ have a common factor with the same zero set.  
\end{proof}

\indent 
The zero sets of one variable polynomials in $\r[x]$ also have a very nice property.  The Fundamental Theorem of Algebra guarantees that not only are they are finite, but their maximum size is determined by the degree of the polynomial. 
Letting $deg_x p$ denote the degree of $x$ of $p(x,y)\in\r[x,y]$ and $deg_y p$ denote the degree of $y$ of $p(x,y)$, the following shows that, in a similar manner, the degree of the variables can, under certain circumstances, determine the type of zero set of a two variable polynomial in $\r[x,y].$

\begin{thm}
Let $p(x, y) \in \r[x, y].$  If $deg_x p$ or $deg_y p$ is odd, then $p(x,y)$ has an uncountabley infinite number of zeros.
\end{thm}
 
\begin{proof}
Suppose $deg_x p$ is odd.  Write $p(x,y)$ in descending powers of $x$ as $p(x,y)=q_n(y)x^{n}+q_{n-1}(y)x^{n-1}+\cdots+q_0(y)x^{0}$ where each $q_i\in \r[y].$ Fix  a  real number $y_0$ and consider the polynomial $p(x,y_0).$  It is a polynomial of the variable $x$ and has degree $n$ which we assumed to be odd.  Thus it must have at least one real root.  This is true for each real number $y_0$ so $p(x,y)$ must have an uncountabley infinite number of zeros.  If $deg_y p$ is odd, then write $p(x,y)$ in descending powers of $y$ and fix an $x_0 \in \r$.  Then use the same argument as before.  
\end{proof}

\begin{rem} This shows that if a polynomial in $\r[x,y]$ is to have a finite, discrete zero set then $deg_x p$ and $deg_y p$ must both be even.
\end{rem}  

If $F$ is not algebraically closed the result of Theorem \ref{thm: algebraically closed} breaks down.
Take $F=\r$,  $n=2$, and consider the polynomials $p(x,y)=x^2+y^2$ and $g(x,y)=x^4+y^4$ in $\r[x,y]$.  The zero set of both polynomials is the single point $(0,0)$, yet they do not have a common factor.  
The main result shows that with certain conditions on the zero set 
 satisfied, the result of Theorem \ref{thm: algebraically closed} can be extended to a large class of polynomials in $\r[x,y]$.  In addition, when we replace $\r$ with the division algebra of the Quaternions there are in fact examples of two polynomials sharing an infinite zero set that do not share a factor, one of which we show in Section \ref{section: Quaternions}.

\section{Main result}\label{section: main result}

In the setting of one variable polynomials it is natural use the division algorithm to establish divisibility. Given specific polynomials we can perform long division and look for the remainder to be zero. This idea can be extended to two variable polynomials.
Write $p$ and $g$ in $\r[x,y]$ by descending powers of $x$ so as to view them as polynomials of the one variable $x$ with coefficients in $\r[y]$ and then do long division just as if they were one variable polynomials, the result of which yields an equation similar to the division algorithm.\\

For instance, if $g(x,y)=5x^3-2$ and $p(x,y)=x-3y$ dividing $g$ by $p$ gives\\ \\
\centerline{$\polylongdiv{5x^3-2}{x-3y}$}\\ \\ \\
Thus $g(x,y)=p(x,y)(5x^2+15yx+45y^2)+135{y^3}-2.$\\ \\
If $p(x,y)=2x^4-3x$ and $g(x,y)=yx^2+yx$ then dividing $p$ by $g$ gives\\ \\
\centerline{$\polylongdiv{2x^4-3x}{yx^2+yx}$}\\ \\

So   $2{x^4}-3x=(y{x^2}+yx)(\frac{2x^2}{y}-\frac{2x}{y}+\frac{2}{y})-5x.$\\ \\

 Note that in the first example both the quotient and remainder are polynomials of two variables but in the latter the quotient is a polynomial with coefficients that are rational functions of the variable $y$.  In general, if $p$ and $g$ are polynomials of two variables over the real numbers and $n=deg_x{p}$ then\\ 
 \begin{equation}\label{eqn: dividing1}
g=pq+r
\end{equation}

where $q$ and $r$ are polynomials with coefficients in $\r(y)$ and $deg_x{r}\leq n.$\\ 
 
If $n$= $deg{_x} p$ there are many examples of polynomials in $\r[x,y]$ that have a zero set crossed by an infinite number of horizontal  lines in at least $n$ distinct places.  Any polynomial that has its zero set on the curve $y=x$ or $y=x^n$ naturally satisfies the conditions.  Using this idea of long division we can show that these conditions ensure these polynomials will share the factor $y-x$ or $y-x^n$ respectively with any polynomial that has the same zero set. Furthermore, Theorem \ref{thm: horizontal lines} uses Equation \ref{eqn: dividing1} as a foundation to show that any two polynomials sharing a zero set with these conditions have a common factor with the same zero set.

\begin{thm}\label{thm: horizontal lines}
Let $p(x, y), g(x,y) \in \r[x, y]$ be two polynomials that have the same zero set and 
$n$=    $deg{_x} p$.  
If there are an infinite number of horizontal lines lines that intersect the zero set in at least $n$ distinct places,  then $p$ and $g$ have a common factor with the same zero set.
\end{thm}

\begin{proof}  Let $n$=$deg{_x} p$ and assume that there are an infinite number of parallel lines that intersect the zero set of $p$ and $g$ in at least $n$ distinct places.  
Dividing $g(x,y)$ by $p(x,y)$ gives
\begin{equation}\label{eqn: dividing}
g(x,y)=q(x,y)p(x,y)+r(x,y)
\end{equation}
 where $r(x,y)$ and $q(x,y)$ are polynomials of the variable $x$ with coefficients in $
 \r(y)$ and $deg{_x} r \leq {n-1}.$ 
 If we let $h(y)$ be the common denominator of the coefficients in $q$ and $r$ and multiply both sides of (\ref{eqn: dividing}) by h we then have \\ \\
 \centerline{$hg$=${\tilde q}p+{\tilde r}$} \\  \\where $\tilde q=hq$ and $\tilde r=hr$ are polynomials of two variables and since $h\in \r[y]$ $deg_x{\tilde r}\leq {n-1}.$  This means that  ${\tilde r}(x,y)=r_{n-1}(y)x^{n-1}+ r_{n-2}(y)x^{n-2}+\cdots +r_{0}(y)$ where $r_i(y)\in \r[y].$ 
Now if $(x_0, y_0)$ is a point in the zero set of $p$ and $g$ then $g(x_0, y_0)=p(x_0, y_0)=0$.  Thus $0={\tilde q}(x_0,y_0)\cdot0 + {\tilde r}(x_0, y_0)$.  This implies that ${\tilde r}(x_0,y_0)=0.$ 
 
By assumption there are an infinite number of horizontal lines that intersect the zero set in at least $n$ distinct places. Let $y_0$ be the $y$ coordinate of the points on one of these horizontal lines. It follows that there are  at least $n$ distinct $x$ values such that \[ {\tilde r}(x,y_0)=0.\]
But since $deg{_x} {\tilde r} \leq {n-1},$ it follows that \[ r_{n-1}(y_0)= r_{n-2}(y_0)= r_{n-3}(y_0)= \cdots = r_{0}(y_0)=0. \]
Now, there are infinitely many values $y_0$ for which the above relationship holds, and since all $r_i$'s are polynomials of one variable it follows that they are identically equal to zero. Thus $ {\tilde r}(x,y)=0,$ and therefore
\begin{equation}\label{eqn:remainder zero}
hg={\tilde q} p.
\end{equation}
  
This equation tells us that $p$ divides $hg.$  Since $\r[x,y]$ is a UFD either $p$ is irreducible or it can be written as the product of irreducible factors.  We will now show that in either case $p$ and $g$ have a common factor with the same zero set.\\
\indent If $p$ is irreducible it is also prime and must therefore divide either $h$ or $g$.  First suppose it divides $h.$  This means that $h={\tilde h}p$ for some ${\tilde h} \in \r[x,y].$  Since $h$ is a polynomial of the one variable $y$ this implies that $p$ is also a polynomial of the one variable $y.$  But the zero set of a polynomial of one variable is a finite number of horizontal lines in $\r^2$, and it is therefore impossible for there to exist infinitely many horizontal lines that intersect this zero set nontrivially. It follows that $p$ cannot divide $h$ and so must divide $g.$

If $p$ is not irreducible then it can be written as the product of irreducible factors $p_1 \cdots p_k$ for some positive integer $k$ and the zero set of $p$ would be the union of the zero sets of these factors.  Since each of these factors must also be prime, they each either divide $g$ or divide $h.$  Now if they all divide $h$ then the zero set of $p$ would be a finite number of horizontal lines, of which it is impossible for an infinite number of horizontal lines to cross.  Thus one of these factors must divide $g$.  If any of the other irreducible factors of $p$ divide $h$ then the zero set of each factor that is nontrivial is a horizontal line.  
The only way for the zero set of $g$ to contain all the points on a horizontal line is for $g$ to have 
a factor that is only in terms of the variable $y$.  It follows that each of these distinct factors must also divide $g$.  Thus $p$ and $g$ share a common factor with the same zero set.  

 \end{proof}

\noindent Please note that by writing $p(x,y)$ in descending powers of $y$ one can similarly prove

\begin{thm}\label{thm: vertical lines}
Let $p(x, y), g(x,y) \in \r[x, y]$ be two polynomials that have the same zero set and 
$n$=  $deg{_y} p$.  
If there are an infinite number of vertical lines that intersect the zero set in at least $n$ distinct places,  then $p$ and $g$ have a common factor with the same zero set.
\end{thm}

Recall that the degree of a term of a two variable polynomial is the sum of the exponents of the variables in the term and the degree is the maximum of the degrees of the terms which we will denote by $deg [p]$.  Now based on these Theorems we give a slightly more general statement.

\begin{cor}\label{cor:general}
Let $p(x, y), g(x,y) \in \r[x, y]$ be two polynomials that have the same zero set and 
$n$= $deg [p]$.  
If there are an infinite number of parallel lines that intersect the zero set in at least $n$ distinct places,  then $p$ and $g$ have a common factor with the same zero set.
\end{cor}
\begin{proof} Let $\frac{a}{b}$ be the slope of the parallel(we may assume they are not vertical or horizontal) lines that intersect the zero set of the two polynomial lines in at least $n$ distinct places. Consider the change of variables

\[  u=bx+ay \] and \[ v=ax-by .\]
This linear transformation is invertible and the two polynomials in the new variables $u,v$ will share the same zero set. Let us denote these polynomials by $P(u,v)$ and $G(u,v).$
It is clear that $ deg{_u}P(u,v)=deg{_v}P(u,v)=n.$  The parallel lines that intersect the zero set of $p(x,y)$ and $g(x,y)$ are mapped into horizontal lines in the $u,v$ plane. We are now under the conditions of Theorem  \ref{thm: horizontal lines}. It follows that $P(u,v)$ divides $G(u,v)$ and applying the inverse linear transformation we conclude that $p$ and $g$ have a common factor with the same zero set.

\end{proof}
Note that no polynomial with a finite zero set will fall under the above theorems.  This may in part explain why the polynomials $p(x,y)=x^2+y^2$ and $g(x,y)=x^4+y^4$ in $\r[x,y]$ do not have a common factor.  There are also polynomials with infinite zero sets that are left out.  One example is any polynomial with a zero set on the curve ${x^4}+{y^4}=1.$  In this case, $n=4$ but one can show algebraically that this curve will be crossed at no more than two distinct places by any line.  

\section{On the non-commutative case}\label{section: Quaternions}
Let us consider the case of polynomials over the division algebra of quaternions $\mathbb{H}.$ If we look at   \[ p(x,y)=xy-yx, \]
then the zero set is the set of all commuting pairs of quaternions in $\mathbb{H}^2.$ According to \cite{MC}, the polynomial \[ g(x,y)=x^2y+y^2x - 2xyx,\]
has the same exact zero set as $p(x,y).$  The following theorem shows that $p$ and $g$ do not have a common factor with the same zero set.

\begin{thm}
Let $p(x,y)=xy-yx$ and $g(x,y)=x^2y+y^2x - 2xyx.$  Then $p$ and $g$ do not have a common factor with the same zero set. 
\end{thm}

\begin{proof}
We will first show that $p$ is irreducible over $\mathbb{H}$ and then show that $p$ does not divide $g$ thus implying that $p$ and $g$ do not share a common factor.  

Suppose $p$ is the product of two monomials.  Then $xy-yx=(ax+by+c)(Ax+By+C)$ where $a, b, c, A, B, C, \in \mathbb{H}.$ Multiplying the right side out gives \\

\centerline{$xy-yx=aAx^2+aBxy+aCx+bAyx+bBy^2+bCy+cAx+cBy+cC.$}
\bigskip
Then $aA=0$, $aB=1$, and $bA= -1$.  If $aA=0$, then since $\mathbb{H}$ has no zero divisors either $a=0$ or $A=0.$   However this contradicts that fact that $aB=1$ and $bA= -1$.  It follows that $p$ cannot be the product of two monomials and is therefore irreducible.  

Now suppose that $p$ divides $g$.  Then either $g(x,y)=p(x,y)h(x,y)$ or $g(x,y)=h(x,y)f(x,y)$ for some $h(x,y) \in\mathbb{H}[x.y].$  Since $h(x,y)$ cannot be a constant suppose that it is linear, that $h(x,y)=ax+by+c$.  If $g=ph$ then $x^2y+y^2x - 2xyx=(xy-yx)(ax+by+c)$.  Multiplying the right side out gives\\

\centerline{$x^2y+y^2x - 2xyx=axyx+bxy^2+cxy-ayx^2-byxy-cyx.$}
\bigskip
Since there is no $x^2y$ term it follows that $g\neq ph.$  If $g= hp$ then a similar argument shows that there cannot be a $y^2x$ term.  It follows that $p$ does not divide $g$ and that they have no common factor with the same zero set.

\end{proof}

\noindent  Department of Mathematics, Davinci Academy, Ogden, UT 84401,\\ USA,  
\  \textbf{e-mail:} \ {\tt micki.balaich@davinciacademy.org}  
\medskip
 
\noindent Department of Mathematics,
 Weber State University, Ogden, UT  84408, USA,   
\  \textbf{e-mail:} \ {\tt mihailcocos@weber.edu}

\end{document}